\documentclass[11 pt]{amsart}

\usepackage{amsmath,amsthm,amssymb,amscd,dsfont}
\usepackage[colorlinks=true, pdfstartview=FitV, linkcolor=blue, citecolor=red,urlcolor=blue]{hyperref}

\newtheorem{thm}{Theorem}[section]
\newtheorem{cor}[thm]{Corollary}

\newtheorem{ass}[thm]{Assumption}

\newtheorem{prop}[thm]{Proposition}

\theoremstyle{definition}
\newtheorem{definition}[thm]{Definition}
\theoremstyle{remark}
\newtheorem{rem}[thm]{Remark}

\numberwithin{equation}{section}

\def\be#1 {\begin{equation} \label{#1}}
\newcommand{\ee}{\end{equation}}

\newcommand{\N}{\mathbb N}

\newcommand{\BMO}{\textrm{BMO}}

\let\Re=\relax
\DeclareMathOperator{\Re}{Re}

\def\Xint#1{\mathchoice
   {\XXint\displaystyle\textstyle{#1}}%
   {\XXint\textstyle\scriptstyle{#1}}%
   {\XXint\scriptstyle\scriptscriptstyle{#1}}%
   {\XXint\scriptscriptstyle\scriptscriptstyle{#1}}%
   \!\int}
\def\XXint#1#2#3{{\setbox0=\hbox{$#1{#2#3}{\int}$}
     \vcenter{\hbox{$#2#3$}}\kern-.5\wd0}}

\def\aver#1{\Xint-_{#1}}

\begin{document}

\begin{abstract} 
This paper can be considered as a sequel of \cite{BS} by Bernicot and Samoyeau, where the authors have proposed a general way of deriving Strichartz estimates for the Schr\"odinger equation from a dispersive property of the wave propagator.
It goes through a reduction of $H^1-\BMO$ dispersive estimates for the Schr\"odinger propagator to $L^2-L^2$ microlocalized (in space and in frequency) dispersion inequalities for the wave operator. 
This paper aims to contribute in enlightening our comprehension of how dispersion for waves imply dispersion for the Schr\"odinger equation. More precisely, the hypothesis of our main theorem encodes dispersion for the wave equation in an uniform way, with respect to the light cone.
In many situations the phenomena that arise near the boundary of the light cone are the more complicated ones. The method we present allows to forget those phenomena we do not understand very well yet. %in the litterature ?
%The method we present first shows how to obtain a $L^2-L^2$ microlocalized dispersion inequality from mild assumption on the wave operator. 
The second main step shows the Strichartz estimates with loss of derivatives we can obtain under those assumptions.
The setting we work with is general enough to recover a large variety of frameworks (infinite metric spaces, Riemannian manifolds with rough metric, some groups, ...) where the lack of knowledge of the wave propagator is a restraint to our understanding of the dispersion phenomenon.
\end{abstract}



\title[Strichartz estimates under weak dispersion]{Strichartz estimates with loss of derivatives under a weak dispersion property for the wave operator}

\author{Valentin Samoyeau}
\address{Valentin Samoyeau - Universit\'e de Nantes \\ Laboratoire Jean Leray
\\ 2, rue de la Houssini\`ere
44322 Nantes cedex 3, France}
\email{valentin.samoyeau@univ-nantes.fr}
\urladdr{http://www.math.sciences.univ-nantes.fr/~samoyeau/}

\thanks{V. Samoyeau's research is supported by Centre Henri Lebesgue (program
"Investissements d'avenir" --- ANR-11-LABX-0020-01).}
\thanks{This work benefits from and fits into the ERC project FAnFArE no. 637510.}

\numberwithin{equation}{section}

\date{\today}

\maketitle

\begin{quote}
\footnotesize\tableofcontents
\end{quote}

\section{Introduction}\label{section_introduction}

% This paper can be considered as a sequel of Bernicot and Samoyeau \cite{BS}, where the authors have proposed a general way of deriving Strichartz estimates for the Schr\"odinger equation from a dispersive property of the wave propagator.
% They shew a reduction of $L^1-L^\infty$ dispersive estimates for the Schr\"odinger propagator to a $L^2-L^2$ microlocalized (in space and in frequency) dispersion inequality for the wave operator. 

The family of so-called Strichartz estimates is a powerful tool to study nonlinear Schr\"odinger equations. Those estimates give a control of the size of the solution to a linear problem in term of the size of the initial data. 
The ``size'' notion is usually given by a suitable functional space $L^p_t L^q_x$. 
Such inequalities were first introduced by Strichartz in \cite{Strichartz} for Schr\"odinger waves on the 
Euclidean space. They were then extended by Ginibre and Velo in \cite{GV} (and the endpoint is due to 
Keel and Tao in \cite{KT}) for the propagator operator
associated with the linear Schr\"odinger equation in $\mathbb R^d$. So for an initial data $u_0$, we 
are interested in controlling $u(t,\ldotp) = e^{it\Delta}u_0$ which is the solution of the linear Schr\"odinger equation:
$$\begin{cases}
   &i\partial_tu + \Delta u=0 \\
   &u_{|t=0}=u_0.
  \end{cases}$$
It is well-known that the unitary group $e^{it\Delta}$ satisfies the following inequality: $$\|e^{it\Delta}u_0\|_{L^pL^q([-T,T] \times \mathbb R^d)} \leq C_T \|u_0\|_{L^2(\mathbb R^d)}$$
for every pair $(p,q)$ of admissible exponents which means : $2\leq p,q \leq \infty$, $(p,q,d)\neq(2,\infty,2)$, and 
\begin{equation}\label{admissible} %\tag{$p,q$}
 \frac 2p + \frac dq = \frac d2.
\end{equation}
The Strichartz estimates can be deduced via a $TT^*$ argument from the dispersive estimates 
\begin{equation}\label{disp_L1_Linfty}
 \|e^{it \Delta}u_0\|_{L^{\infty}( \mathbb R^d)} \lesssim |t|^{-\frac d2} \|u_0\|_{L^1(\mathbb R^d)}.
\end{equation}

If $\sup_{T>0}C_T<+\infty$, we will say that a global-in-time Strichartz estimate holds. Such a global-in-time 
estimate has been proved by Strichartz for the flat Laplacian on $\mathbb R^d$ 
while the local-in-time estimate is known in several geometric situations where the manifold is non-trapping 
(asymptotically Euclidean, conic, or hyperbolic, Heisenberg group); see \cite{BT,Bouclet, HTW,ST,BGX} or 
with variable coefficients \cite{RZ,T}. 
% The finite volume of the manifold and the presence of trapped geodesics appear to limit the extent to 
% which dispersion can occur. 

The situation for compact manifolds presents a new difficulty, since considering the
constant initial data on the torus $u_0= 1\in L^2(\mathbb T)$ yields a contradiction in \eqref{disp_L1_Linfty} for large time.

% The situation is similar for the non-compact case in the presence of elliptic (stable)
% non-degenerate periodic orbits of the geodesic flow: as remarked
% by M. Zworski, some loss must occur as far as Strichartz estimates are concerned;
% and moreover, that no Strichartz estimates can be true globally in time in the presence of such orbits.

Burq, G\'erard, and Tzvetkov \cite{BGT} and Staffilani and Tataru \cite{ST} proved that Strichartz estimates
hold on a compact manifold $\mathcal M$ for finite time if one considers regular data $u_0\in W^{1/p,2}(\mathcal M)$. Those are 
called ``with a loss of derivatives'': 
$$\|e^{it\Delta}u_0\|_{L^pL^q([-T,T] \times \mathcal M)} \leq C_T \|u_0\|_{W^{1/p,2}(\mathcal M)}.$$
An interesting problem is to determine for specific situations, 
which loss of derivatives is optimal (for example the work of Bourgain \cite{Bourgain} on the flat torus and \cite{TT} of Takaoka and Tzvetkov). 
For instance, the loss of $\frac 1p$ derivatives in \cite{BGT} is shown to be optimal in the case of the sphere.

% Numerous recent works aim also to obtain such Strichartz estimates with a loss of 
% derivatives in various situations, for example corresponding to a Laplacian operator 
% on a smooth domain with boundary condition (Dirichlet or Neumann); see the works of 
% Anton \cite{Anton}, Blair-Smith-Sogge \cite{BSS} and Blair-Ford-Herr-Marzuola 
% \cite{BFHM}.
% 
% The case of infinite manifolds (with boundary) with one
% trapped orbit was considered by Christianson in \cite{Christianson} where a larger
% loss of derivatives of $1/p+\varepsilon$ is obtained. There the author allows to
% perturb the Laplacian by a smooth potential.

An important remark is that, by Sobolev embedding, the loss of $2/p$ derivatives is straightforward.
Indeed, by Sobolev embedding, we have $W^{\frac 2p,2} \hookrightarrow L^q$ since $\frac 2p - \frac d2 = 0-\frac dq$ so that 
\begin{equation} \|e^{it\Delta}u_0\|_{L^q} \lesssim \|e^{it\Delta}u_0\|_{W^{\frac 2p,2}} \leq \|u_0\|_{W^{\frac 2p,2}} \label{eq:tri} \end{equation}
and taking the $L^p([-T,T])$ norm yields $$\|e^{it\Delta}u_0\|_{L^pL^q} \leq C_T \|u_0\|_{W^{\frac 2p,2}}.$$
Therefore Strichartz estimates with loss of derivatives are interesting for a loss smaller than $2/p$.

Let us now set the general framework of our study.
We consider $(X, d, \mu)$ a metric measured space equipped with a nonnegative $\sigma$-finite Borel measure $\mu$. 
We assume moreover that $\mu$ is Alfhor regular, that is there exist a dimension $d$, and two 
absolute positive constants $c$ and $C$ such that for all $x \in X$ and $r > 0$
\begin{equation} \label{ah}
 c r^{d} \leq \mu(B(x, r)) \leq C r^{d},
\end{equation}
where $B(x,r)$ denote the open ball with center $x \in X$ and radius $0<r<\text{diam}(X)$.
% As a consequence, there exists a homogeneous dimension $d>0$, such that
% \begin{equation} \label{dd} \forall x \in X \virg \forall r>0 \virg \forall t \geq 1 \virg \mu(B(x,tr))\lesssim t^d \mu(B(x,r)). \end{equation}
Thus we aim our results to apply in numerous cases of metric spaces such as open subsets of $\mathbb R^d$, smooth $d$-manifolds, some fractal sets, Lie groups, Heisenberg group, \dots

Keeping in mind the canonical example of the Laplacian operator in $\mathbb R^d$: $\Delta = \sum_{1\leq j \leq d} \partial_j^2$, we will be more general in the following sense:
we consider a nonnegative, self-adjoint operator $H$ on $L^2=L^2(X,\mu)$ densely defined, which means that its domain 
$$ {\mathcal D}(H):=\{f \in L^2,\, Hf \in L^2\}$$
is supposed to be dense in $L^2$. 

One of the motivation of our paper is to study the connection between the wave equation and the Schr\"odinger equation. We define the wave propagator $\cos(t \sqrt H)$ as follows: 
for any $f\in L^2$, $u(t, \ldotp):=t\mapsto \cos(t \sqrt H)f$ is the unique solution of the linear wave problem
\begin{equation} \label{wave_equation} \begin{cases}
   &\partial_t^2u + Hu=0 \\
   &u_{|t=0} = f\\
   &\partial_t u_{|t=0} = 0.
  \end{cases}
  \end{equation}
One can find the explicit solutions of this problem in \cite{Fol} for the Euclidean 
case and in \cite{Berard} for the Riemannian manifold case through precise formula 
for the kernel of the wave propagator. Up to our knowledge, those explicit solutions are not available in our abstract setting. 
It would be of great interest to be able to compute exact expression of the solution of the wave equation in such a general case.
The remarkable property of this operator comes from its finite speed propagation: 
for any two disjoint open subsets $U_1,U_2 \subset X$, and any functions $f_i \in L^2(U_i)$, $i=1,2$, then
\begin{equation}\label{speed_propagation}
 \langle \cos(t\sqrt H)f_1,f_2 \rangle =0
\end{equation}
for all $0<t<d(U_1,U_2)$. If $\cos(t\sqrt H)$ is an integral operator with kernel $K_t$, 
then \eqref{speed_propagation} simply means that $K_t$ is supported in the ``light cone'' 
$\mathcal D_t:=\{(x,y) \in X^2,\ d(x,y)\leq t\}$.
We assume that $H$ satisfies \eqref{speed_propagation}. In \cite{CS}, Coulhon and Sikora proved that this property is equivalent to the Davies-Gaffney estimates
\begin{equation}\label{davies_gaffney_estimates}%\tag{$DG$}
\|e^{-tH}\|_{L^2(E) \to L^2(F)} \lesssim e^{- \frac{d(E,F)^2}{4t}}
\end{equation}
for any two subsets $E$ and $F$ of $X$, and $t>0$. \newline
It is known that $-H$ is the generator of a $L^2$-holomorphic semigroup $(e^{-tH})_{t\geq0}$ 
(see \cite{Davies}).
We will also assume that the heat semigroup $(e^{-tH})_{t\geq0}$ satisfies the typical upper estimates 
(for a second order operator): for every $t>0$ the operator $e^{-tH}$ admits a kernel $p_t$ with
\begin{equation}\tag{$DU\!E$}
|p_{t}(x,x)|\lesssim
\frac{1}{\mu(B(x,\sqrt{t}))}, \quad \forall~t>0,\,\mbox{a.e. }x\in X.\label{due}
\end{equation}

It is well-known that such on-diagonal pointwise estimates self-improve into the 
full pointwise Gaussian estimates (see \cite[Theorem 1.1]{Gr1} or \cite[Section 4.2]{CS} e.g.)
\begin{equation}\tag{$U\!E$}
|p_{t}(x,y)|\lesssim
\frac{1}{\mu(B(x,\sqrt{t}))}\exp
\left(-c\frac{d(x,y)^2}{t}\right), \quad \forall~t>0,\, \mbox{a.e. }x,y\in
 X.\label{UE}
\end{equation}
One can find in \cite{BS} and the references therein some examples where the previous estimates hold.

When dealing with Schr\"odinger equation on a manifold or a more general metric space, the $L^1-L^\infty$ estimate \eqref{disp_L1_Linfty} seems out of reach. 
In \cite{BS}, the authors show how to replace it by a $H^1-\BMO$ estimate, with the Hardy space $H^1$ and the Bounded Mean Oscillations space $\BMO$ both adapted to the semigroup. 
We do not recall the definition of those spaces here, but refer to \cite{BS} for more details.

%We are going to use a partition of unity adapted to the semigroup setting. We chose it over the usual $C^\infty_0$ function because it is more convenient to compute the norms in term of semigroup in our framework (see Subsection \ref{subsection_functional_calculus}). 
For any integer $m \geq 0$ and $x \in \mathbb R_+$ we set $\psi_m(x)=x^m e^{-x}$. It forms a family of smooth functions that vanish at $0$ (except when $m=0$) and infinity, which allows us to consider 
a smooth partition of unity, using holomorphic functionnal calculus (and requiring $\mathbb C^\infty_0$-calculus). %Moreover $\|\psi_m\|_{L^\infty(\mathbb R_+)} \lesssim 1$. 
\\
The main assumption of our work is the following
\begin{ass}\label{assumption_cos}
 There exist $\kappa  \in (0, \infty]$ and an integer $m$ such that for every $s\in(0,\kappa)$ the wave propagator $\cos(s \sqrt H)$ at time $s$ satisfies the following dispersion property
 \begin{equation}\label{hypothesis}
  \|\cos(s\sqrt H)  \psi_m(r^2H)\|_{L^2(B) \to L^2(\widetilde B)} \lesssim \left( \frac{r}{r+s} \right)^{\frac{d-1}{2}},
 \end{equation}
 for any two balls $B$, $\widetilde B$ of radius $r>0$.
\end{ass}
This estimate is microlocalized in the physical space due to the balls $B$ and $\widetilde B$ at scale $r$ and in frequency at scale $\frac 1r$ through $\psi_m(r^2H)$, thus respecting the Heisenberg uncertainty principle.
The parameter $\kappa$ is linked to the geometry of the space $X$ (its injectivity radius for instance). 

In the Euclidean space $\mathbb R^d$, the $L^2(B)-L^2(\widetilde B)$ dispersion phenomenon seems only to depend on the distance $d(B, \widetilde B)$. Indeed, the intuition is that, in an isotropic medium a wave propagates the same way in all the directions.
That is what leads us to think that Assumption \ref{assumption_cos} could be proved without using a pointwise explicit formula of its kernel, but with a more general approach, using functional tools only, that could be extend to other settings. 
To our knowledge the study of such behavior is not known and could be a good direction to investigate. 

We mentioned that the finite speed propagation property \eqref{speed_propagation} gives us the idea that after time $s$ the solution to the wave problem \eqref{wave_equation} with initial data supported in a ball of radius $r$ is supported in a ball of radius $r+s$.
Given that $r\leq s$ (otherwise $L^2$ functional calculus yields Assumption \ref{assumption_cos}) and the fact that waves propagate the same way in all directions in an isotropic medium, if we cover the sphere of radius $r+s$ by $N \simeq \left( \frac{r+s}{r} \right)^{d-1}$ balls of radius $r$ and use Theorem \ref{thm_calcul_fonctionnel}, 
we can conjecture that the term $\left( \frac{r}{r+s} \right)^{\frac{d-1}{2}}$ is the natural dispersion one can hope for such waves, if we look for a uniform estimate (depending only on $r$, $s$).  
%However we know that the estimate can be improved far away from the light cone. 

Indeed we also emphasize that Assumption \ref{assumption_cos} is weaker than the one in \cite{BS}, namely: \\
There exist $\kappa  \in (0, \infty]$ and an integer $m$ such that for every $s\in(0,\kappa)$ we have
\begin{equation}\label{hyp_cossH_intro}
 \|\cos(s\sqrt H) \psi_m(r^2H)\|_{L^2(B) \to L^2(\widetilde B)} \lesssim \left( \frac{r}{r+s} \right)^{\frac{d-1}{2}} \left( \frac{r}{r+ |s-L|} \right)^{\frac{d+1}{2}},
\end{equation}
where $L=d(B, \widetilde B)$, which describes more precisely the dispersion inside the light cone. \\
However \eqref{hyp_cossH_intro} can be difficult to prove in an abstract setting. That is why we are interested in proving what Assumption \ref{assumption_cos} could imply as far as Strichartz estimates are concerned. 
Estimate \eqref{hypothesis} should indeed be much easier to prove in concrete examples. 
%Indeed in the present paper we do not assume any knowledge of the behavior of the wave propagator near the boundary of the light cone. 
For more on Assumption \ref{assumption_cos}, see Subsection \ref{subsection_motivations}. 
Consequently the results we obtain will be weaker too.
We recall Theorem 1.3 from \cite{BS} in order to compare it with our Theorem \ref{thm_2}.

\begin{thm}[\cite{BS}]\label{thm_BS} Suppose \eqref{ah} with $d>1$, \eqref{due} and
Assumption \eqref{hyp_cossH_intro} with $\kappa\in(0,\infty]$. Two cases occur: 
\begin{itemize}
\item if $\kappa=\infty$ then we have Strichartz estimates without loss of derivatives;
\item if $\kappa<\infty$ then for every $\varepsilon>0$, every $0<h\leq1$ with
$h^2\leq |t|<h^{1+\varepsilon}$ we have Strichartz estimates with loss of $\frac{1+\varepsilon}{p}$ derivatives.
\end{itemize}
\end{thm}
To prove this, the authors first reduced the $H^1-\BMO$ estimation to a microlocalized $L^2-L^2$ estimate, and then showed how dispersion for the wave propagator implies dispersion for the Schr\"odinger group. 
Theorem \ref{demo_hyp_Hmn} is playing that role in the present paper.

Our main theorem follows the routine of \cite{BS} to deduce Strichartz inequalities from $L^2-L^2$ estimates.
\begin{thm}\label{thm_2}
 Assume \eqref{ah} with $d>2$, \eqref{due}, and Assumption \ref{assumption_cos}. Then for every $2\leq p \leq +\infty$ and $2\leq q < +\infty$ satisfying $$\frac 2p + \frac{d-2}{q} = \frac{d-2}{2},$$
 and every solution $u(t,\ldotp)=e^{itH}u_0$ of the problem
 $$\begin{cases}
  &i \partial_t u + Hu=0\\
  &u_{|t=0}=u_0,
 \end{cases}$$ we have 
 \begin{itemize}
  \item if $\kappa = \infty$, then $u$ satisfies local-in-time Strichartz estimates with loss of derivatives 
  \begin{equation}\label{strichartz_estimates_intro_global}
   \|u\|_{L^p([-1,1], L^q)} \lesssim \|u_0\|_{W^{2(\frac12-\frac 1q),2}};
  \end{equation}
  \item if $\kappa < \infty$, then for every $\varepsilon >0$, $u$ satisfies local-in-time Strichartz estimates with loss of derivatives
  \begin{equation}\label{strichartz_estimates_intro}
   \|u\|_{L^p([-1,1], L^q)} \lesssim \|u_0\|_{W^{\frac{1+\varepsilon}{p}+2(\frac 12-\frac 1q),2}}.
  \end{equation}
 \end{itemize}
\end{thm}
We would like to point out that the straightforward loss of derivatives given by Sobolev embeddings when $$\frac 2p + \frac{d-2}{q} = \frac{d-2}{2}$$ is $$\frac 2p + 1-\frac 2q.$$ 
Thus the loss is nontrivial here. For more on the loss of derivatives, see Remarks \ref{remark_loss_of_derivatives} and \ref{remark_loss_of_derivatives_2}. 
It is interesting to see how a weak dispersion property on the wave propagator implies dispersion for the Schr\"odinger operator. 

The idea of the proof here is similar to the one in \cite{BS}. More particularly it is due to a precise tracking of the constants in some key estimations (from \cite{KT} for instance).
\\
The aim of this paper is to give a better understanding of how dispersion for the wave propagator implies dispersion for the Schr\"odinger equation, and what Strichartz inequalities ensue in 
some contexts, where we do not have precise dispersive estimates on the wave propagators. 
In other words if one can compute, even inaccurate, information about the wave propagator in general settings, it would allow to have some knowledge of the Schr\"odinger equation in that framework.

The organization of the paper is as follow: 
In Section \ref{section_definitions} we set the notations used throughout the paper and recall some preliminary facts concerning the semigroup, Hardy and $\BMO$ spaces, as well as some motivations of our hypothesis. 
Then Section \ref{section_proofs} is dedicated to the proofs of the Theorems. 

%statement of the result, organisation of the paper, references to related work, etc...

\section{Preliminaries}\label{section_definitions}

\subsection{Notations}

We denote $\text{diam}(X):= \sup \limits_{x,y \in X} d(x,y)$ the diameter of a metric space $X$. 
For $B(x,r)$ a ball ($x \in X$ and $r>0$) and any parameter $\lambda >0$, we denote $\lambda B(x,r) := B(x, \lambda r)$ the dilated and concentric ball. 
As a consequence of \eqref{ah}, a ball $B(x, \lambda r)$ can be covered by $C \lambda^d$ balls of radius $r$, uniformly in $x\in X$, $>0$ and $\lambda >1$ ($C$ is a constant only depending on the ambient space). 
If no confusion arises, we will note $L^p$ instead of $L^p(X, \mu)$ for $p \in [1, +\infty]$. 
For $s>0$ and $p \in [1, +\infty]$, we denote by $W^{s,p}$ the Sobolev space of order $s$ based on $L^p$, equipped with the norm $$\|f\|_{W^{s,p}} := \|(1+H)^{\frac s2}f\|_{L^p}.$$
\\
We will use $u \lesssim v$ to say that there exists a constant $C$ (independent of the important parameters) such that $u \leq C v$ and $u \simeq v$ to say that both $u \lesssim v$ and $v \lesssim u$. 
If $\Omega$ is a set, $\mathds1_{\Omega}$ is the characteristic function of $\Omega$, defined by
$$\mathds1_{\Omega}(x) = \begin{cases}&1 \textrm{ if } x \in \Omega \\
&0 \textrm{ if } x \notin \Omega.
   \end{cases}$$

Throughout the paper, unless something else is explicitly mentioned, we assume that $d>2$ and that \eqref{ah}, \eqref{due}, \eqref{davies_gaffney_estimates}, and Assumption \ref{assumption_cos} are satisfied.

\subsection{The heat semigroup and associated functional calculus} \label{subsection_functional_calculus}

We consider a nonnegative, self-adjoint operator $H$ on $L^2=L^2(X,\mu)$ densely defined. 
We recall the bounded functional calculus theorem from \cite{RS}: 
\begin{thm}\label{thm_calcul_fonctionnel}
$H$ admits a $L^\infty$-functional calculus on $L^2$: if $f \in L^{\infty}(\mathbb R_+)$, then we may consider the
operator $f(H)$ as a $L^2$-bounded operator and
$$\|f(H)\|_{L^2 \to L^2}\leq \|f\|_{L^{\infty}}.$$
\end{thm}

From the Gaussian estimates of the heat kernel \eqref{UE} and the analyticity of the
semigroup (see \cite{CCO}) it comes that for every integer $m\in{\mathbb N}$ and every $t>0$, the operator $\psi_{m}(tH)$ has a kernel $p_{m,t}$ also satisfying upper
Gaussian estimates:
\begin{equation}
|p_{m,t}(x,y)|\lesssim
\frac{1}{\mu(B(x,\sqrt{t}))}\exp
\left(-c \frac{d(x,y)^2}{t}\right), \quad \forall~t>0,\, \mbox{a.e. }x,y\in
 X.\label{UEm}
\end{equation}

We now give some basic results about the heat semigroup thanks to our assumptions. The detailed proofs can be found in Section $2$ of \cite{BS}.

\begin{prop}\label{prop_continuite_semigroupe} Under \eqref{ah} and \eqref{due}, the
heat semigroup 
%is pointwisely bounded by the Hardy-Littlewood maximal operator and
is uniformly bounded in every $L^p$-spaces for $p\in[1,\infty]$; more precisely for every $f\in L^p$, we have
% \sup_{t>0} \left\|e^{-tH} f \right\|_{L^\infty(B(x_0,\sqrt{t}))} \lesssim
%\calM(f)(x_0)  \quad \textrm{and} \quad
$$ \sup_{t>0} \| e^{-tH} f \|_{L^p} \lesssim \|f\|_{L^p}.$$
Moreover, for $m \in \mathbb N$ and $t>0$, since $\psi_m(tH)$ also satisfies \eqref{due} we have $$ \sup_{t>0} \| \psi_m(tH) \|_{L^p \to L^p} \lesssim 1.$$
\end{prop}

Let us now define some tools for the Littlewood-Paley theory we need in the sequel. For all $\lambda >0$ we set
$$\varphi(\lambda) := \int_{\lambda}^{+\infty} \psi_m(u) \frac{du}{u},$$
$$\tilde \varphi(\lambda) := \int_0^{\lambda} \psi_m(v) \frac{dv}{v} = \int_0^1 \psi_m(\lambda u) \frac{du}{u}.$$

\begin{rem}
 Notice that $\varphi$ is, by integration by parts, a finite linear combination of
 functions $\psi_{k}$ for $k\in\{0,..,m\}$. Moreover for every $\lambda>0$,  
  $$\tilde\varphi(\lambda)+\varphi(\lambda)=\int_0^{+\infty}u^{m-1}e^{-u}du=\Gamma(m)=\textrm{constant}.$$
\end{rem}

The following theorem will be useful to estimate the $L^p$-norm through the heat
semigroup: 

\begin{thm}\label{thm_LP} Assume \eqref{ah} and \eqref{due}. For every integer
$m\geq 1$ and all $p\in(1,\infty)$, we have
\begin{align*}
 \|f\|_{L^p} & \simeq  \|\varphi(H)f\|_{L^p}+\left\| \left(\int_0^1
|\psi_{m}(uH)f|^2\frac{du}u\right)^{\frac 12} \right\|_{L^p}. 
\end{align*}
So if $q\geq 2$
\begin{align*}
 \|f\|_{L^q} & \lesssim  \|\varphi(H)f\|_{L^q}+ \left(\int_0^1
\|\psi_{m}(uH)f\|_{L^q}^2\frac{du}u\right)^{\frac 12}. 
\end{align*}
\end{thm}

Such a result can be seen as a semigroup version of the Littlewood-Paley
characterization of Lebesgue spaces. A proof of this theorem can be found in \cite{BS} (look for Theorem 2.8 in \cite{BS}).

\subsection{Hardy and BMO spaces}

We now define atomic Hardy spaces adapted to our situation (dictated by a semigroup) using the construction introduced in \cite{BZ}. Again we sum up the definitions and properties we need without proofs.
A more detailed explanation with proofs is provided in \cite{BS}.
% 
% Let $\mathcal Q$ be the family of all balls of $X$ 
% $$\mathcal Q := \{ B(x,r) \virg x\in X \virg r>0 \}.$$
% We define $\left( B_Q \right)_{Q \in \mathcal Q}$ the family of operators $$\forall
% Q \in \mathcal Q \virg B_Q:=(1-e^{-r^2H})^M,$$ where $r$ is the radius of the ball
% $Q$ and $M$ is a large enough integer.
% %(large enough: $M\geq \min(\frac34+\frac{3d}{8},3)$ is sufficient). 
% Those operators are bounded on $L^{2}$ uniformly in $r$. 
% Indeed, by expanding, $B_Q$ 
% is a finite linear combination of operators $e^{-kr^2H}$ with $k \in \{ 0,
% \ldots, M \}$ and Theorem \ref{thm_calcul_fonctionnel} gives $$\|e^{-kr^2H}\|_{L^2
% \to L^2} \leq \|x \mapsto e^{-kr^2x}\|_{L^{\infty}(\mathbb R_+)} \leq 1,$$ because
% $H$ is nonnegative. 

Let $M$ be a large enough integer.

\begin{definition}\label{def_atome}
 A function $a \in L^1_{loc}$ is an atom associated with the ball $Q$ of radius $r$ if there exists
a function $f_Q$ whose support is included in $Q$ such that $a=(1-e^{-r^2H})^M(f_Q)$, with 
 $$\|f_Q\|_{L^{2}(Q)} \leq ( \mu(Q) )^{-\frac{1}{2}}.$$
\end{definition}
That last condition allows us to normalize $f_Q$ in $L^1$. Indeed by the
Cauchy-Schwarz inequality $$\|f_Q\|_{L^1}\leq \|f_Q\|_{L^2(Q)} \mu(Q)^{\frac 12}
\leq 1.$$
Moreover, $(1-e^{-r^2H})^M$ is bounded on $L^{1}$ so every atom is in $L^{1}$ and they are also
normalized in $L^1$: 
\begin{equation}\label{thm_atome_L1} \sup_{a} \|a\|_{L^1} \lesssim 1, 
\end{equation}
where we take the supremum over all the atoms. 
% Indeed, consider an atom 
% $a=B_Q(f_Q)=(1-e^{-r^2H})^Mf_Q$ with suitable function $f_Q$ supported on a ball $Q$.
% By the binomial theorem, $B_Q$ behaves like $e^{-kr^2H}$. So Proposition
% \ref{prop_continuite_semigroupe} gives 
% $$\|a\|_{L^1(X)}=\|B_Q(f_Q)\|_{L^1(X)} \leq \sum_{k=1}^M \binom Mk
% \|e^{-kr^2H}f_Q\|_{L^1} \lesssim \|f_Q\|_{L^1} \lesssim 1.$$

We may now define the Hardy space by atomic decomposition

\begin{definition}\label{def_hardy}
 A measurable function $h$ belongs to the atomic Hardy space $H^1_{ato}$, which will
be denoted $H^1$, if there exists a decomposition $$h=\sum_{i \in \N} \lambda_i a_i
\quad \mu- \textrm{a.e.}$$
 where $a_i$ are atoms and $\lambda_i$ real numbers satisfying $$\sum_{i \in \N}
|\lambda_i| < + \infty.$$
 We equip the space $H^1$ with the norm $$\|h\|_{H^1}:= \inf \limits_{h=\sum_i
\lambda_i a_i}\sum_{i\in \N}|\lambda_i|,$$
where we take the infimum over all the atomic decompositions.
\end{definition}

For a more general definition and some properties about atomic spaces we refer to
\cite{Ber,BZ}, and the references therein.  From \eqref{thm_atome_L1}, we deduce
\begin{cor} \label{injH1L1} The Hardy space is continuously embedded into $L^1$:
$$ \|f\|_{L^1} \lesssim  \|f\|_{H^1}.$$
From \cite[Corollary 7.2]{BZ}, the Hardy space $H^1$ is also a Banach space.
\end{cor}

We refer the reader to \cite[Section 8]{BZ}, for details about the problem of
identifying the dual space $(H^1)^*$ with a $\BMO$ space. For a $L^\infty$-function,
we may define the $\BMO$ norm
$$ \|f\|_{\BMO} := \sup_{Q} \left(\aver{Q} |(1-e^{-r^2H})^M(f)|^2 \, d\mu \right)^{1/2},$$
where the supremum is taken over all the balls $Q$ of radius $r>0$. If $f\in L^\infty$ then $(1-e^{-r^2H})^M(f)$ is
also uniformly bounded (with respect to the ball $Q$), since the heat semigroup is
uniformly bounded in $L^\infty$ (see Proposition \ref{prop_continuite_semigroupe})
and so $\|f\|_{\BMO}$ is finite.

\begin{definition} The functional space $\BMO$ is defined as the closure 
$$\BMO:= \overline{\left\{ f\in L^\infty + L^2,\ \|f\|_{\BMO}<\infty \right\} }$$
for the $\BMO$ norm.
\end{definition}

Following \cite[Section 8]{BZ}, it comes that $\BMO$ is continuously embedded into
the dual space $(H^1)^*$ and contains $L^\infty$: 
$$ L^\infty \hookrightarrow \BMO \hookrightarrow (H^1)^*.$$
Hence
\begin{equation}\label{H1*doncBMO}
 \|T\|_{H^1 \to (H^1)^*} \lesssim \|T\|_{H^1 \to \BMO},
\end{equation}
and we have the following interpolation result: 
\begin{equation}\label{H1*doncBMO2}
 \forall \theta \in (0,1), \quad (L^2,\BMO)_{\theta} \hookrightarrow (L^2, (H^1)^*)_{\theta}.
\end{equation}

The following interpolation theorem between Hardy spaces and Lebesgue spaces is essential in our study.
\begin{thm}\label{interpolation}
For all $\theta \in (0,1)$, consider the exponent $p\in(1,2)$ and
$q=p'\in(2,\infty)$ given by
$$\frac 1p=\frac{1-\theta}{2}+\theta \quad \textrm{ and }\quad \frac{1}{q} =
\frac{1-\theta}{2}.$$
Then (using the interpolation notations), we have
$$(L^2, H^1)_{\theta}=L^p \quad \textrm{and} \quad (L^2, (H^1)^*)_{\theta}
\hookrightarrow L^{q},$$
if the ambient space $X$ is non-bounded and
$$ L^p \hookrightarrow L^2+(L^2, H^1)_{\theta} \quad \textrm{and} \quad L^2 \cap
(L^2, (H^1)^*)_{\theta}  \hookrightarrow L^{q},$$
if the space $X$ is bounded.

The same results hold replacing $(H^1)^*$ by $\BMO$ thanks to \eqref{H1*doncBMO2}.
\end{thm}

% In the situation of bounded space (with a finite measure), interpolation is more
% delicate since the previous result does not give a complete characterization of
% $L^p$ as an intermediate space. We have the following:
% \begin{thm}\label{interpolation_finie}  Assume that the space is bounded (or
% equivalently that $\mu(X)<+\infty$) and consider a self-adjoint operator $T$
% satisfying the following boundedness conditions 
% $$\begin{cases}
%  &\|T\|_{L^2 \to L^2} \lesssim 1\\
%  &\|T\|_{H^1 \to (H^1)^*} \lesssim A<+\infty\\
%  &\|T\|_{L^p  \to L^2} \lesssim B<+\infty \quad \textrm{for } p\in (1,2)
% \end{cases},$$ 
% then $T$ is bounded from $L^p$ to $L^{p'}$ with
% $$ \|T\|_{L^p \to L^{p'}} \lesssim  B+A^{\frac{1}{p}-\frac{1}{p'}}.$$
% The same result holds with $\BMO$ instead of $(H^1)^{*}$ by \eqref{H1*doncBMO2}.
% \end{thm}

\begin{rem}
 We will not mention the case of a bounded space $X$ in the proofs since interpolation is more delicate in that case. One can find the corresponding interpolation theorem (Theorem 2.17 in \cite{BS}) and check that the results apply in that case.
\end{rem}

\subsection{Motivation of the hypothesis}\label{subsection_motivations}

This section is dedicated to the motivation of Assumption \ref{assumption_cos}.

As we said in the Introduction, this hypothesis is weaker than the one in \cite{BS}, namely
\begin{equation}\label{assumption_cos_BS}
 \|\cos(s\sqrt H) \psi_m(r^2H)\|_{L^2(B) \to L^2(\widetilde B)} \lesssim \left( \frac{r}{r+s} \right)^{\frac{d-1}{2}} \left( \frac{r}{r+ |s-L|} \right)^{\frac{d+1}{2}},
\end{equation}
where $L=d(B, \widetilde B)$.
Therefore in all the situations where \eqref{assumption_cos_BS} is satisfied, we can assure that Assumption \ref{assumption_cos} is valid. 
When we have a good knowledge of the wave propagator, we can also affirm that Assumption \ref{assumption_cos} holds. This is the case thanks to a parametrix in \cite{Berard} in the following cases: 
\begin{itemize}
 \item The Euclidean spaces $\mathbb R^d$ with the usual Laplacian $H=-\Delta=-\sum_{j=1}^d \partial^2_j$;
 \item A compact Riemannian manifold of dimension $d$ with the Laplace-Beltrami operator;
 \item A smooth non-compact Riemannian $d$-manifold with $C^\infty_b$-geometry and Laplace-Beltrami operator;
 \item The Euclidean space $\mathbb R^d$ equipped with the measure $d\mu=\rho dx$ and $H = -\frac{1}{\rho} \textrm{div}(A\nabla)$, where $\rho$ is an uniformly non-degenerate function and $A$ a matrix with bounded derivatives;
\end{itemize}

Moreover we can check that for a non-trapping asymptotically conic manifold with $H=-\Delta + V$ the assumption holds. Therefore we recover, with another proof, the result of \cite{HZ}.

However in \cite{ILP} the authors proved that for the Laplacian inside a convex domain of dimension $d\geq 2$ in $\mathbb R^d$, there was a loss of $s^{\frac 14}$ in the dispersion, namely
\begin{equation}\label{disp_ILP}
\|\cos(s\sqrt H) \psi_m(r^2H)\|_{L^2(B) \to L^2(\widetilde B)} \lesssim \left( \frac{r}{r+s} \right)^{\frac{d-1}{2}+\frac 14} \left( \frac{r}{r+ |s-L|} \right)^{\frac{d+1}{2}}.
\end{equation}
This loss indicates a difficulty when dealing with boundaries of a domain. 
The authors used oscillatory integrals techniques and a careful study of the reflections on the boundary of the domain. 

A remark of J.-M. Bouclet to get around the use of a parametrix leads us to investigate Klainerman's commuting vectorfields method. It can be found in detail in \cite{Kla} and \cite{Sogge}. 
Briefly, if one can find enough vectorfields commuting with the wave operator, using a version of Sobolev inequalities, also know as Klainerman-Sobolev inequalities, one can obtain dispersion estimations for the wave propagator. 
In our setting, we would obtain (see \cite[Remark 1.4]{Sogge}) the following dispersion property: 
\begin{equation}\label{disp_kla}
\|\cos(s\sqrt H) \psi_m(r^2H)\|_{L^2(B) \to L^2(\widetilde B)} \lesssim \left( \frac{r}{r+s} \right)^{\frac{d-1}{2}} \left( \frac{r}{r+ |s-L|} \right)^{\frac 12}. 
\end{equation}
It is very close to our Assumption \ref{assumption_cos}, but it takes into account the dispersion inside the light cone. In that sense, it is intermediate between our Assumption \ref{assumption_cos} and estimate \eqref{assumption_cos_BS}. 
A question we would like to pursue investigating is to find enough well-suited vectorfields to apply this method in generals settings.
The framework inwhich one is interested in verifying \eqref{disp_kla} is when $H$ is a given by divergence form, namely $H=-\textrm{div}(A\nabla)$. 
When $A=\textrm{I}_d$ the identity matrix of size $d$, $H$ is the usual Laplacian. In this case and the one where $A$ has $C^{1,1}$ coefficients, Klainerman obtained in \cite{Kla} a dispersion property of the form \eqref{disp_kla}. 
It is not new since it was already proven in \cite{Smith}. But the novelty in \cite{Kla} is to get around the use of a parametrix.

%poser des questions ouvertes dans les motivations

%justifier quand l'hyp est valable, comment on le montre, et ainsi donner tous les exemples où ce qu'on fait est vrai et donne qqc

\section{Proofs of the Theorems}\label{section_proofs}

This section is dedicated to the proofs of the announced result. It is divided into two main theorems. The first one shows which $L^2-L^2$ dispersion property we can recover 
thanks to Assumption \ref{assumption_cos}. In the second theorem we obtain Strichartz inequalities using such dispersive estimates. 
We recall that our goal is to investigate which properties (in terms of Strichartz inequalities) for the Schr\"odinger operator can be deduced from a weak assumption on the wave operator. 

\subsection{Dispersive estimates for the Schr\"odinger operator}

The main theorem of this section is the following
\begin{thm}\label{demo_hyp_Hmn}
 Assume $d>2$, $m \geq \lceil \frac{d}{2} \rceil$, and that Assumption \ref{assumption_cos} is satisfied, then for
 all balls $B$, $\widetilde B$ of radius $r>0$ and all $0<h\leq1$
 \begin{equation}\label{hyp_1}
   \|e^{itH} \psi_{m'}(h^2H) \psi_m(r^2H)\|_{L^2(B) \to L^2(\widetilde{B})} \lesssim \frac{r^d}{|t|^{\frac{d-2}{2}} h^{2}},
 \end{equation}
 for all $m' \geq 0$ and where $0<|t|\leq 1$ if $\kappa = +\infty$ and $h^2 \leq |t| \leq h^{1+\varepsilon}$ if $\kappa < +\infty$.
 %Hence $e^{itH} \psi_{m'}(h^2H)$ satisfies $H_m\left ( t^{-\frac{d-2}{2}}h^{-2} \right)$.
\end{thm}

This shows how dispersion for the wave propagator implies dispersion for the Schr\"odinger group.
The main tool to link those two operators is Hadamard's transmutation formula
\begin{equation} \label{hadamard}
 \forall z \in \mathbb C,\ \Re(z)>0,\ e^{-zH} = \int_0^{+\infty} \cos(s \sqrt H) e^{- \frac{s^2}{4z}} \frac{ds}{\sqrt{\pi z}}.
\end{equation}

\begin{proof}
 Let $B$, $\widetilde B$ be balls with radius $r > 0$. 
 We start our proof with some easy reductions.
 
 First remark that we can restrict ourselves to prove the theorem for $h \leq r$. Indeed if the theorem is true for $h \leq r$ then  for all $h >r$
 \begin{align*}
  \|e^{itH} \psi_{m'}(h^2H) \psi_m(r^2H)\|_{L^2(B) \to L^2(\widetilde{B})} &= \frac{r^{2m}}{(\frac{h^2}{2}+r^2)^m} \|e^{itH} \psi_{m'}(\frac{h^2}{2}H) \psi_{m'}((\frac{h^2}{2}+r^2)H)\|_{L^2(B) \to L^2(\widetilde{B})} \\
  &\lesssim \left( \frac rh \right)^{2m} \|e^{itH} \psi_{m'}(\frac{h^2}{2}H) \psi_{m'}((\frac{h^2}{2}+r^2)H)\|_{L^2(B_{\rho}) \to L^2(\widetilde{B}_{\rho})}
 \end{align*}
 where $\rho = \frac{h^2}{2}+r^2 > r$, $B_{\rho} = \frac{\rho}{r}B$ and $\widetilde B_{\rho} = \frac{\rho}{r}\widetilde B$ are of radius $\rho$. Since $\frac{h^2}{2}+r^2 \geq \frac{h^2}{2}$ we then obtain
 $$\|e^{itH} \psi_{m'}(h^2H) \psi_m(r^2H)\|_{L^2(B) \to L^2(\widetilde{B})} \lesssim \left( \frac rh \right)^{2m} \frac{\rho^d}{|t|^{\frac{d-2}{2}}h^2}.$$
 We conclude using $\rho \lesssim h$ and $$\frac{r^{2m}h^d}{h^{2m}} =r^d \left(\frac{r}{h}\right)^{2m-d}\leq r^d.$$
 
 Moreover we only need to prove the theorem for $m'=0$. Indeed if we show
 $$ \|e^{itH} \psi_{0}(h^2H) \psi_m(r^2H)\|_{L^2(B) \to L^2(\widetilde{B})} \lesssim \frac{r^d}{|t|^{\frac{d-2}{2}} h^{2}}$$
 then for all $m' \geq 0$ we have
 \begin{align*}
   \|e^{itH} \psi_{m'}(h^2H) \psi_m(r^2H)\|_{L^2(B) \to L^2(\widetilde{B})} &= \left( \frac hr \right)^{2m'} \|e^{itH} \psi_{0}(h^2H) \psi_{m+m'}(r^2H)\|_{L^2(B) \to L^2(\widetilde{B})} \\
   &\lesssim \frac{r^d}{|t|^{\frac{d-2}{2}} h^{2}}
 \end{align*}
 since $h\leq r$

 Finally it is sufficient to consider $r^2\leq t$ because if $r^2>t$ then by bounded functional calculus we have
 $$\|e^{itH} \psi_{m'}(h^2H) \psi_m(r^2H)\|_{L^2(B) \to L^2(\widetilde{B})} \lesssim 1 \leq \left(\frac{r^2}{|t|}\right)^{\frac{d-2}{2}} \leq \frac{r^d}{|t|^{\frac{d-2}{2}}h^2}.$$
 
 In summary, we fix $h\leq r$, $m'=0$, and $r^2\leq t$.
 
 In order to avoid nonzero bracket terms in the forthcoming integrations by parts, we introduce a technical function $\chi \in C^{\infty}(\mathbb R_+)$ such that 
 $$\begin{cases}
 &0\leq \chi \leq 1\\
 &\chi(x)=1 \text{ if } x \in [0, \frac{|t|}{r}]\\
 &\chi(x)=0 \text{ if } x \in [2\frac{|t|}{r}, +\infty]
 \end{cases}.$$ 
 Moreover we have $\forall n \in \mathbb N$, $\forall x \in \mathbb R_+$, $|\chi^{(n)}(x)| \lesssim \left( \frac{r}{|t|} \right)^n$.
 Thus we split \eqref{hadamard} into 
 \begin{equation}\label{hadamard_split}
 e^{-zH} = \int_0^{+\infty} \chi(s)\cos(s \sqrt H) e^{- \frac{s^2}{4z}} \frac{ds}{\sqrt{\pi z}} + \int_0^{+\infty} (1-\chi(s))\cos(s \sqrt H) e^{- \frac{s^2}{4z}} \frac{ds}{\sqrt{\pi z}}.
 \end{equation}
 
 We treat the first term by integrations by parts. Making $2n$ integration by parts (with $n$ to be determined later) we get
\begin{align*}
 &\int_0^{\infty}\cos(s \sqrt H) \psi_m(r^2H) \chi(s) e^{-\frac{s^2}{4z}} ds = \\
 \int_0^{\infty}&\cos(s \sqrt H) r^{2n} \psi_{m-n}(r^2H) \sum_{k=0}^{2n} \chi^{(2n-k)}(s) e^{-\frac{s^2}{4z}}\left(c_k \frac{s^k}{z^k} + \ldots + c_{n-2\lfloor \frac n2 \rfloor} \frac{s^{k-2\lfloor \frac k2 \rfloor}}{z^{k-\lfloor \frac k2 \rfloor}} \right) ds,
\end{align*}
with $(c_i)_i$ being numerical constants playing no significant role.
Keeping the extremal terms (one when $k=0$ and two when $k=2n$) we have to estimate
$$\int_0^{2\frac{|t|}{r}} \| \cos(s\sqrt H) \psi_{m-n}(r^2H) \|_{L^2(B) \to L^2(\widetilde B)} r^{2n} \left( \left( \frac{r}{|t|} \right)^{2n} + \frac{s^{2n}}{|t|^{2n}} + \frac{1}{|t|^n} \right) \frac{ds}{\sqrt{|t|}}.$$
By continuity of our operators $$\| \cos(s\sqrt H) \psi_{m-n}(r^2H) \|_{L^2(B) \to L^2(\widetilde B)} \lesssim 1,$$ we can estimate
$$\int_0^{2\frac{|t|}{r}} \left( \frac{r^2}{|t|} \right)^{2n} \frac{ds}{\sqrt{|t|}} \lesssim \left( \frac{r^2}{|t|} \right)^{2n-\frac 12} \quad \textrm{and} \quad \int_0^{2\frac{|t|}{r}} \frac{r^{2n}}{|t|^{n}} \frac{ds}{\sqrt{|t|}} \lesssim \left( \frac{r^2}{|t|} \right)^{n-\frac 12}.$$
Using \eqref{assumption_cos} we have
$$\int_0^{2\frac{|t|}{r}} \left(\frac{r}{r+s}\right)^{\frac{d-1}{2}} \left( \frac{rs}{|t|} \right)^{2n} \frac{ds}{\sqrt{|t|}} \leq \int_0^{2\frac{|t|}{r}} \frac{r^{\frac{d-1}{2} + 2n}}{|t|^{2n+\frac 12}} s^{2n-\frac{d-1}{2}} ds \simeq \left( \frac{r^2}{|t|} \right)^{\frac{d-2}{2}}.$$
Thus, the intermediate terms having the same behaviour, for large enough $n$ $$\left \| \int_0^{\infty}\cos(s \sqrt H) \psi_m(r^2H) \chi(s) e^{-\frac{s^2}{4z}} \frac{ds}{\sqrt z} \right \|_{L^2(B) \to L^2(\widetilde B)} \lesssim \left( \frac{r^2}{|t|} \right)^{\frac{d-2}{2}}.$$
Moreover, since $h \leq r$ we have $$\left( \frac{r^2}{|t|} \right)^{\frac{d-2}{2}} \leq \frac{r^d}{|t|^{\frac{d-2}{2}}h^{2}}.$$

To estimate the second term in \eqref{hadamard_split}, we treat separately the cases $s< \kappa$ and $s > \kappa$.
$$\int_0^{+\infty} (1-\chi(s))\cos(s \sqrt H) e^{- \frac{s^2}{4z}} \frac{ds}{\sqrt{\pi z}} = \int_{\frac{|t|}{r}}^{\kappa} \cos(s \sqrt H) e^{- \frac{s^2}{4z}} \frac{ds}{\sqrt{\pi z}} + I_\kappa,$$
where $$I_\kappa = \left \{\begin{array}{ccc}
                    &0 \quad &\textrm{if} \quad \kappa = +\infty \\
                    &\int_\kappa^{+\infty} \cos(s \sqrt H) e^{- \frac{s^2}{4z}} \frac{ds}{\sqrt{\pi z}} \quad &\textrm{if} \quad \kappa <+\infty
                   \end{array}\right..$$
We use the exponential decay for $s > \kappa$. Noting $z=h^2-i|t|$, and using the $L^2$-boundedness of the $\cos(s \sqrt H) \psi_m(r^2H)$ operator: 
\begin{align*}
 \left \| \int_\kappa^{+\infty} \cos(s \sqrt H) \psi_m(r^2H) e^{-\frac{s^2}{4z}} \frac{ds}{\sqrt z} \right \|_{L^2(B) \to L^2(\widetilde B)} & \lesssim \int_\kappa^{+\infty} e^{-\frac{s^2}{8} \Re \frac 1z} e^{-\frac{s^2}{8} \Re \frac 1z} \frac{ds}{\sqrt{|z|}} \\
 & \lesssim \int_{\frac \kappa8 \sqrt{\Re \frac 1z}}^{+\infty}e^{-u^2}\frac{du}{\sqrt{|z| \Re \frac 1z}} e^{-\frac{\kappa^2\Re \frac 1z}{8}} \\
 & \lesssim \left(\int_0^{+\infty}e^{-u^2}du\right) \frac{\sqrt{|t|}}{h} e^{-\frac{\kappa^2 h^2}{16t^2}}\\
 & \lesssim \left(\frac{h^2}{t^2}\right)^{- N}\frac{\sqrt{|t|}}{h}
\end{align*}
for all $N\geq1$ as large as we want and where we used $|z| \simeq |t|$ and $\Re \frac 1z \geq \frac{h^2}{2t^2}$. 
Moreover $$ \frac{|t|^{2N+\frac12}}{h^{2N+1}} \leq \frac{h^d}{|t|^{\frac{d-2}{2}}h^{2}} \leq \frac{r^d}{|t|^{\frac{d-2}{2}}h^{2}}$$ as soon as
$|t|^{2N + \frac{d-2}{2} + \frac 12 } \leq h^{2N+d-1}$ that is $|t| \leq h^{1+ \frac{\frac{d-1}{2}}{2N+\frac{d-1}{2}}}$.
Which is true since $$|t|\leq h^{1+\varepsilon} \leq h^{1+ \frac{\frac{d-1}{2}}{2N+\frac{d-1}{2}}}$$ for $N$ large enough. 
\begin{rem}
We point out that this is the only moment we use that $|t| \leq h^{1+\varepsilon}$. 
That is why we do not need it when $\kappa = +\infty$ since this term does not step in. 
Therefore the loss of derivatives in Theorem \ref{thm_2} is better when $\kappa = + \infty$.
\end{rem}

We use Assumption \ref{assumption_cos} when $s < \kappa$. 
%We treat separately the cases $\alpha >1$, $\alpha <1$, and $\alpha =1$. 
Indeed it yields
$$\left \| \int_{\frac{|t|}{r}}^{\kappa} \cos(s \sqrt H) \psi_m(r^2H) e^{-\frac{s^2}{4z}} \frac{ds}{\sqrt z} \right \|_{L^2(B) \to L^2(\widetilde B)} \lesssim \int_{\frac{|t|}{r}}^\kappa \left( \frac{r}{r+s} \right)^{\frac{d-1}{2}} e^{-\frac{s^2}{4} \Re \frac 1z} \frac{ds}{\sqrt{|t|}}.$$
When $\frac{d-1}{2} > 1$ (i.e. $d>3$) we have
\begin{align*}
 \int_{\frac{|t|}{r}}^\kappa \left( \frac{r}{r+s} \right)^{\frac{d-1}{2}} \frac{ds}{\sqrt{|t|}} &\leq \frac{r^{\frac{d-1}{2}}}{\sqrt{|t|}} \int_{\frac{|t|}{r}}^\infty s^{-\frac{d-1}{2}} ds \lesssim \frac{r^{\frac{d-1}{2}}}{\sqrt{|t|}} \left( \frac{|t|}{r} \right)^{-\frac{d-1}{2}+1} \\
 &\leq \frac{r^{d-2}}{|t|^{\frac{d-2}{2}}}  \leq \frac{r^d}{|t|^{\frac{d-2}{2}}h^{2}} h^2 \leq \frac{r^d}{|t|^{\frac{d-2}{2}}h^{2}}
\end{align*}
since $h^2 \leq 1$. \newline
When $\frac{d-1}{2} <1$ (i.e. $d<3$), since $\Re \frac 1z \gtrsim \frac{h^2}{t^2}$ we have
\begin{align*}
 \int_{\frac{|t|}{r}}^\kappa \left( \frac{r}{r+s} \right)^\frac{d-1}{2} e^{-c\frac{s^2h^2}{t^2}} \frac{ds}{\sqrt{|t|}} &\lesssim \frac{r^\frac{d-1}{2}}{\sqrt{|t|}} \int_{\frac hr}^\infty e^{-u^2} \left( \frac{|t|u}{h} \right)^{-\frac{d-1}{2}} \frac{|t|}{h} du \\
 & \lesssim \frac{r^{\frac{d-1}{2}} h^{\frac{d-3}{2}}}{|t|^{\frac{d-2}{2}}} \int_0^\infty u^{-\frac{d-1}{2}} e^{-u^2}du \\
 & \lesssim \frac{r^d}{|t|^{\frac{d-2}{2}}h^2}\frac{h^{\frac{d-3}{2}}h^2}{r^{\frac{d+1}{2}}} \leq \frac{r^d}{|t|^{\frac{d-2}{2}}h^2}
\end{align*}
since $h \leq r$. \newline 
When $\frac{d-1}{2}=1$ (i.e. $d=3$) we have 
\begin{align*}
 \int_{\frac{|t|}{r}}^\kappa \frac{r}{r+s}e^{-\frac{s^2}{4} \Re \frac 1z} \frac{ds}{\sqrt{|t|}} &\lesssim \frac{r}{\sqrt{|t|}} \int_{\frac hr}^{\kappa \frac{h}{|t|}} \left( \frac{|t|u}{h} \right)^{-1} e^{-u^2} \frac{|t|}{h} du \\
 &\leq \frac{r}{\sqrt{|t|}} e^{-\frac{h^2}{2r^2}}\frac{r}{h}\int_0^{+\infty} e^{-\frac{u^2}{2}} du \lesssim \frac{r^2}{\sqrt{|t|} h} \left(\frac hr \right)^{-1} = \frac{r^3}{\sqrt{|t|} h^2} = \frac{r^d}{|t|^{\frac{d-2}{2}} h^2}.
\end{align*}
In the end, summing all the parts up, we have $$\|e^{itH}\psi_{m'}(h^2H) \psi_m(r^2H)\|_{L^2(B) \to L^2(\widetilde B)} \lesssim \frac{r^d}{|t|^{\frac{d-2}{2}} h^2}.$$
 
\end{proof}

\subsection{Strichartz inequalities}

To obtain Strichartz estimates we are going to use Theorem 1.1 of \cite{BS}, which we recall here with a slight modification in assuming \eqref{ah}, namely

\begin{thm} \label{thm1_BS} Assume \eqref{ah} with \eqref{due}. 
Consider a self-adjoint and $L^2$-bounded operator $T$ (with $\|T\|_{L^2 \to L^2} \lesssim 1$), 
which commutes with $H$ and satisfies 
\begin{equation}\label{disp} \tag{$H_m(A)$}
 \|T \psi_{m}(r^2 H)\|_{L^2(B_r)\to L^2(\widetilde{B_r})} \lesssim A \mu(B_r)^{\frac 12} \mu(\widetilde{B_r})^{\frac12}
\end{equation}
for  some $m \geq \frac d2 $. 
Then $T$ is bounded from $H^1$ to $BMO$ and from $L^p$ to $L^{p'}$ for $p\in(1,2)$ with 
$$ \|T\|_{H^1 \to BMO} \lesssim A \qquad \textrm{and} \qquad \|T\|_{L^p \to L^{p'}} \lesssim A^{\frac{1}{p}-\frac{1}{p'}}$$
if the ambient space $X$ is unbounded and
$$ \|T\|_{H^1 \to BMO} \lesssim \max(A,1) \qquad \textrm{and} \qquad \|T\|_{L^p \to L^{p'}} \lesssim \max(A^{\frac{1}{p}-\frac{1}{p'}},B)$$ 
if the ambient space $X$ is bounded, and where, for the last inequality, we assumed that $\|T\|_{L^p \to L^2}\lesssim B$.
\end{thm}
As we mentioned previously, we do not use the part where $X$ is bounded.
We apply the Theorem with $T=e^{itH} \psi_{m'}(h^2H)$ and $A=|t|^{-\frac{d-2}{2}}h^{-2}$. In view of \eqref{ah} we can reformulate \eqref{disp} (see \cite{BS}) as
\begin{equation}\label{Hmn}
 \|e^{itH}\psi_{m'}(h^2H) \psi_m(r^2H)\|_{L^2(B) \to L^2(\widetilde B)} \lesssim \frac{r^d}{|t|^{\frac{d-2}{2}} h^2}
\end{equation}
which we just proved in the previous section under our assumption. 
Therefore we obtain 
$$ \|e^{itH}\psi_{m'}(h^2H)\|_{H^1 \to BMO} \lesssim |t|^{-\frac{d-2}{2}}h^{-2},$$
and for all $p \in (1,2)$ 
$$\|e^{itH}\psi_{m'}(h^2H)\|_{L^p \to L^{p'}} \lesssim \left[h^{-2} |t|^{-\frac{d-2}{2}}\right]^{\frac{1}{p}-\frac{1}{p'}}.$$
We now recall a slightly modified version of a result of Keel-Tao in \cite{KT}:

\begin{thm}\label{KeelTao}
 If $(U(t))_{t\in \mathbb R}$ satisfies $$\sup \limits_{t\in \mathbb R} \|U(t)\|_{L^2 \to L^2} \lesssim 1$$
 and for some $\sigma >0$ $$\forall t \neq s, \ \|U(t)U(s)^*\|_{H^1 \to \BMO} \leq C |t-s|^{-\sigma}.$$
 Then for all $2\leq p \leq +\infty$ and $2\leq q < +\infty$ satisfying $$\frac 1p + \frac{\sigma}{q} = \frac{\sigma}{2}$$ we have
 $$\|U(t)f\|_{L^p_t L^q_x} \lesssim C^{\frac 12-\frac 1q}\|f\|_{L^2}.$$
\end{thm}

\begin{proof}
 We just sum up the main steps of the proof in \cite{KT} to keep track of the constant in the last estimation.
 \begin{itemize}
  \item By symmetry and a $T^*T$ argument, it suffices to show $$\left | \int_{s<t} \langle U(s)^*F(s), U(t)^*G(t) \rangle ds dt \right | \lesssim C^2 \|F\|_{L^{p'}_t L^{q'}_x} \|G\|_{L^{p'}_t L^{q'}_x}.$$
  \item By the interpolation Theorem \ref{interpolation} we have $$\|U(t)U(s)^*\|_{L^{q'} \to L^q} \lesssim C^{1-\frac 2q}|t-s|^{-\frac 2p}.$$
  \item We conclude by H\"older and Hardy-Littlewood-Sobolev inequalities.
 \end{itemize}
\end{proof}

We use this theorem with $C=\frac{1}{h^2}$ and $\sigma = \frac{d-2}{2}$ to obtain the following result.

\begin{thm}\label{strichartz_localised}
 Under Assumption \ref{assumption_cos}, if $2\leq p \leq +\infty$ and $2\leq q < +\infty$ satisfy $$\frac 2p + \frac{d-2}{q} = \frac{d-2}{2},$$ 
 and $f\in L^2$ and $0<h\leq1$ we have
 \begin{itemize}
  \item if $\kappa = + \infty$ then for all $m' \in \N$ $$\|e^{itH} \psi_{m'}(h^2H)f\|_{L^p([-1,1], L^q)} \lesssim \frac{1}{h^{2(\frac 12-\frac 1q)}} \|\psi_{m'}(h^2H)f\|_{L^2}; $$
  \item if $\kappa < +\infty$ then for all $0< \varepsilon < 1$ and $m' \in \N$ $$\|e^{itH} \psi_{m'}(h^2H)f\|_{L^p([-1,1], L^q)} \lesssim \frac{1}{h^{\frac{1+\varepsilon}{p}} h^{2(\frac 12-\frac 1q)}} \|\psi_{m'}(h^2H)f\|_{L^2}.$$
 \end{itemize}
\end{thm}

\begin{proof}
 The following proof is a slight modification of the one of Theorem 4.2 and 4.3 of \cite{BS}. We rewrite it here for more readability. We only deal with the case $\kappa < + \infty$ since it is more technical. 
 We leave the minor modifications to obtain the case $\kappa = +\infty$ to the readers. \newline
 Fix an interval $J \subset[-1,1]$ of length $|J| = h^{1+\varepsilon}$, $m' \in \N$, and consider $$U(t) = \mathds1_J(t) e^{itH} \psi_{m'}(h^2H).$$
 We aim to apply Theorem \ref{KeelTao} with $C=\frac{1}{h^2}$ and $\sigma = \frac{d-2}{2}$. By functional calculus we have $$\sup \limits_{t\in \mathbb R} \|U(t)\|_{L^2 \to L^2} \lesssim 1.$$
 The estimation \eqref{Hmn} which we proved in Theorem \ref{demo_hyp_Hmn} will lead to the second hypothesis of Theorem \ref{KeelTao}. First 
 \begin{align*}
  U(t)U(s)^*&=\mathds1_J(t) \mathds1_J(s) e^{itH} \psi_{m'}(h^2 H) (e^{isH} \psi_{m'}(h^2H))^* \\
  &=\mathds1_J(t) \mathds1_J(s) e^{i(t-s)H} \psi_{2{m'}}(2h^2H)
 \end{align*}
 because $H$ is self-adjoint. Since $J$ has length equal to $h^{1+\varepsilon}$ then $U(t)U(s)^*$ is vanishing or else $|t-s| \leq h^{1+\varepsilon}$. Hence, by Theorem \ref{demo_hyp_Hmn} we deduce
 $$\|U(t)U(s)^*\|_{H^1 \to (H^1)^*} \lesssim \frac{1}{h^2} \frac{1}{|t-s|^{\frac{d-2}{2}}}.$$
 Up to the change of $2m'$ into $m'$, Theorem \ref{KeelTao} (with $C=h^{-2}$ and $\sigma=(d-2)/2$) then leads to $$\left( \int_J \|e^{itH} \psi_{m'}(h^2H)f\|_{L^q}^p dt \right)^{\frac 1p} \lesssim \frac{1}{h^{2(\frac 12-\frac 1q)}} \|f\|_{L^2}.$$
 We then split $[-1,1]$ into $N \simeq \frac{1}{h^{1+\varepsilon}}$ intervals $J_k$ of length $h^{1+\varepsilon}$ to obtain
 $$\int_{-1}^1 \|e^{itH}\psi_{m'}(h^2H)f\|_{L^q}^p dt \leq \sum_{k=1}^N \int_{J_k} \|e^{itH}\psi_{m'}(h^2H)f\|_{L^q}^p dt \leq N \left( \frac{1}{h^{2(\frac12-\frac 1q)}} \|f\|_{L^2} \right)^p.$$
 Hence $$\|e^{itH} \psi_{m'}(h^2H)f\|_{L^p([-1,1], L^q)} \lesssim \frac{1}{h^{\frac{1+\varepsilon}{p}} h^{2(\frac 12-\frac 1q)}} \|\psi_{m'}(h^2H)f\|_{L^2}.$$
\end{proof}

We are now able to prove Strichartz estimates with loss of derivatives.

\begin{thm}\label{thm_strichartz_estimates}
 If Assumption \ref{assumption_cos} is satisfied. Then for every $2\leq p \leq +\infty$ and $2\leq q < +\infty$ satisfying $$\frac 2p + \frac{d-2}{q} = \frac{d-2}{2},$$
 and every solution $u(t,\ldotp)=e^{itH}u_0$ of the problem
 $$\begin{cases}
  &i \partial_t u + Hu=0\\
  &u_{|t=0}=u_0,
 \end{cases}$$ we have
 \begin{itemize}
  \item if $\kappa = +\infty$, then $u$ satisfies local-in-time Strichartz estimates with loss of derivatives
  \begin{equation}\label{strichartz_estimates_global_corps}
   \|u\|_{L^p([-1,1], L^q)}\lesssim \|u_0\|_{W^{2(\frac 12 - \frac 1q),2}};
  \end{equation}
  \item if $\kappa < +\infty$, then for all $0<\varepsilon < 1$ and $0<h\leq1$, $u$ satisfies local-in-time Strichartz estimates with loss of derivatives
  \begin{equation}\label{strichartz_estimates_local_corps}
   \|u\|_{L^p([-1,1], L^q)}\lesssim \|u_0\|_{W^{\frac{1+\varepsilon}{p} + 2(\frac 12 - \frac 1q),2}}.
  \end{equation}
 \end{itemize}
\end{thm}

\begin{rem}\label{remark_loss_of_derivatives}
 The loss of derivatives in \eqref{strichartz_estimates_local_corps} is interesting when it is lower than the straightforward loss given by Sobolev embeddings. 
 The relation $\frac 2p + \frac{d-2}{q} = \frac{d-2}{2}$ yields $$W^{\frac 2p + 1 - \frac 2q,2} \hookrightarrow L^q.$$
 Thus $$\|e^{itH}u_0\|_{L^q} \lesssim \|e^{itH}u_0\|_{W^{\frac 2p + 1 - \frac 2q,2}} \leq \|u_0\|_{W^{\frac 2p + 1 - \frac 2q,2}}$$
 and taking the $L^p([-1,1])$ norm shows $$\|e^{itH}u_0\|_{L^p([-1,1], L^q)} \lesssim \|u_0\|_{W^{\frac 2p + 1 - \frac 2q,2}}.$$
 That is, the loss of derivatives is interesting when it is less than $\frac 2p + 1 - \frac 2q$.
 Hence, for all $\varepsilon \in (0,1)$, the loss we obtained is  strictly better than the one directly given by Sobolev embeddings. 
 The loss in \eqref{strichartz_estimates_global_corps} is also nontrivial by the same argument. 
\end{rem}

\begin{rem}\label{remark_loss_of_derivatives_2}
 One could work out our estimate with $$\frac 2p + \frac dq = \frac d2.$$
 In order to do so we remark that in \eqref{Hmn} we could write $$\frac{r^d}{t^{\frac{d-2}{2}}h^2} = \left (\frac{r^2}{t} \right)^{\frac d2} \frac{t}{h^2} \leq \left (\frac{r^2}{t} \right)^{\frac d2} \frac{1}{h},$$
 because $t\leq h$. 
 Then the loss of derivatives obtained in \eqref{strichartz_estimates_local_corps} is $\frac{1+\varepsilon}{p} + 1(\frac 12 - \frac 1q)$ that need to be compared to the trivial loss $\frac 2p$.
 Since $\frac 12 - \frac 1q = \frac{2}{dp}$, the loss is less than $\frac 2p$ if an only if $$d \geq \frac{2}{1-\varepsilon}.$$
 That is, as soon as $d>2$, one can find $\varepsilon \in (0,1)$ such that the loss is nontrivial. \\
 We chose to present the previous Theorem in that form because it allows a wider range of exponent $q$. Indeed, on the one hand
 $$p\geq2 \Rightarrow \frac{d-2}{q}= \frac{d-2}{2} - \frac 2p \geq \frac{d-2}{2} - 1$$
 that is $$\frac 1q \geq \frac 12 - \frac{1}{d-2}.$$
 On the other hand $p \geq 2$ and $\frac 2p + \frac dq = \frac d2$ yields $$\frac 1q \geq \frac 12 - \frac 1d,$$
 and for all $d>2$, $$\frac 12 - \frac{1}{d-2} \leq \frac 12 - \frac{1}{d}.$$
 That is why the relation $$\frac 2p + \frac{d-2}{q} = \frac{d-2}{2}$$ gives a wider range for exponent $q$.
\end{rem}

\begin{proof}[Proof of Theorem \ref{thm_strichartz_estimates}]
 Again we only deal with the more difficult case $\kappa < + \infty$. \newline
 We apply Theorem \ref{thm_LP} to $u(t)=e^{itH}u_0$. It leads to
 $$\|u(t)\|_{L^q} \lesssim \|\varphi(H)u(t)\|_{L^q} + \left \| \left( \int_0^{1} |\psi_{{m'}}(s^2H)u(t)|^2 \frac{ds}{s} \right)^{\frac 12} \right \|_{L^q},$$
 with $m'\geq 1$. \newline
 Taking the $L^p([-1,1])$ norm in time of that expression and using Minkowski inequality yields
 $$\|u(t)\|_{L^p([-1,1],L^q)} \lesssim \|\varphi(H)u(t)\|_{L^p([-1,1],L^q)} + \left \| \left( \int_0^{1} \|\psi_{{m'}}(s^2H)u(t)\|_{L^q}^2 \frac{ds}{s} \right)^{\frac 12} \right \|_{L^p([-1,1])}.$$
 Thanks to the Gaussian pointwise estimate of $\varphi(H)$ the first term can be estimated as follow
 $$\|\varphi(H)u(t)\|_{L^p([-1,1],L^q)} \lesssim \|e^{itH}u_0\|_{L^p([-1,1], L^2)} \lesssim \|u_0\|_{L^2} \lesssim \|u_0\|_{W^{\frac{1+\varepsilon}{p}+2(\frac 12-\frac 1q),2}}.$$
 Since $p\geq 2$, Theorem \ref{strichartz_localised} and generalized Minkowski inequality allow to bound the second term 
 \begin{align*}
  \left \| \left( \int_0^{1} \|\psi_{{m'}}(s^2H)u(t)\|_{L^q}^2 \frac{ds}{s} \right)^{\frac 12} \right \|_{L^p([-1,1])}& \lesssim \left( \int_0^{1} \|\psi_{{m'}}(s^2H)u\|_{L^p([-1,1],L^q)}^2 \frac{ds}{s} \right)^{\frac 12} \\
  &\lesssim \left( \int_0^{1} s^{-\frac{1+\varepsilon}{p}-2(\frac 12-\frac 1q)} \|\psi_{{m'}}(s^2H)u_0\|_{L^2}^2 \frac{ds}{s} \right)^{\frac 12} \\
  &\lesssim \left \| \left( \int_0^{1} s^{-\frac{1+\varepsilon}{p}-2(\frac 12-\frac 1q)}|\psi_{{m'}}(s^2H)u_0|^2 \frac{ds}{s} \right)^{\frac 12} \right \|_{L^2}\\
  &\lesssim \|u_0\|_{W^{\frac{1+\varepsilon}{p}+2(\frac 12-\frac 1q),2}},
 \end{align*}
 where we used ${m'} \geq\frac 14[ \frac{1+\varepsilon}{p}+2(1-\frac 2q)]$ since ${m'} \geq 1$ and $\frac{1+\varepsilon}{p}+2(1-\frac 2q)< 2$ and the fact that
 $$s^{-\frac{1+\varepsilon}{p}-2(\frac 12-\frac 1q)} |\psi_{{m'}}(s^2H)|^2 = \psi_{{m'} - \frac14 [\frac{1+\varepsilon}{p}+2(\frac 12-\frac 1q)]}(s^2H) H^{\frac 12 [\frac{1+\varepsilon}{p}+2(\frac 12-\frac 1q)]}.$$
% \textcolor{red}{citer les thm de Littlewood-Paley avec $H$ au début. trouver une meilleure façon de cacher les $\frac{s^2}{2}$ c'est pas joli}\\
 Finally, we get $$ \|u\|_{L^p([-1,1], L^q)} \lesssim \|u_0\|_{W^{\frac{1+\varepsilon}{p}+2(\frac 12-\frac 1q),2}}.$$

\end{proof}

\bibliographystyle{alpha}
\bibliography{bibliographie.bib}
%\addcontentsline{toc}{chapter}{References}

\end{document}